\documentclass{amsart}

\usepackage{amssymb,amsmath, amsthm}
\usepackage{mathrsfs,dsfont}
\usepackage {psfrag,graphicx}
\usepackage[neveradjust]{paralist}
\usepackage{epic,eepic}
\usepackage{url}
\numberwithin{equation}{section}

\theoremstyle{definition}
\newtheorem{ntn}{Notation}[section]

\theoremstyle{plain}

\newtheorem{prp}[ntn]{Proposition}
\newtheorem{thm}[ntn]{Theorem}

\theoremstyle{remark}
\newtheorem{que}[ntn]{Question}
\newtheorem{rmk}[ntn]{Remark}

\newcommand{\fraka}{{\mathfrak{a}}}
\newcommand{\frakb}{{\mathfrak{b}}}

\newcommand{\KK}{\mathds{K}}

\newcommand{\QQ}{\mathds{Q}}

\newcommand{\ZZ}{\mathds{Z}}

\DeclareMathOperator{\conv}{conv}

\DeclareMathOperator{\N}{Newt}

\DeclareMathOperator{\Spec}{Spec}

\DeclareMathOperator{\relint}{relint}

\begin{document}

\title[]
{A counterexample for subadditivity of multiplier ideals on toric varieties}

\author{Jen-Chieh Hsiao}
\address{
Purdue University\\
Dept.\ of Mathematics\\
150 N.\ University St.\\
West Lafayette, IN 47907\\
USA}
\email{jhsiao@math.purdue.edu}
\thanks{The author was partially supported by NSF under grant DMS~0555319 and DMS~0901123.}


\begin{abstract}
We construct a $3$-dimensional complete intersection toric variety on which the subadditivity formula doesn't hold, answering negatively a question by Takagi and Watanabe. 
A combinatorial proof of the subadditivity formula on $2$-dimensional normal toric varieties is also provided.
\end{abstract}

\maketitle

\section{Introduction}

Demailly, Ein and Lazarsfeld \cite{DEL} proved the subadditivity theorem for multiplier ideals on smooth complex varieties, which states
\[ \mathcal{J}(\fraka \frakb) \subseteq \mathcal{J}(\fraka) \mathcal{J}(\frakb) .\]
This theorem is responsible for several applications of multiplier ideals in commutative algebra, in particular to symbolic powers \cite{ELS1}
and Abhyankar valuations \cite{ELS2}.

In a later paper, Takagi and Watanabe \cite{TW04} investigated the extent to which the subadditivity theorem remains true on singular varieties.
They showed that on $\QQ$-Gorenstein normal surfaces, the subadditivity formula holds if and only if the variety is log terminal (\cite{TW04}, Theorem~2.2). Furthermore, they gave an example of a $\QQ$-Gorenstein normal toric threefold on which the formula is not satisfied (\cite{TW04}, Example~3.2).
This led Takagi and Watanabe to ask the following 

\begin{que}\label{subadd} Let $R$ be a Gorenstein toric ring and $\fraka$, $\frakb$ be monomial ideals of $R$. Is it true that 
\[\mathcal{J}(\fraka \frakb) \subseteq \mathcal{J}(\fraka) \mathcal{J}(\frakb) ?\]
\end{que}
The purpose of this article is to provide a counterexample to Question~\ref{subadd}.
We will also give, in section~\ref{2d}, a combinatorial proof of the subadditivity formula on any $2$-dimensional normal toric rings.
The standard notation and facts in \cite{Ful} will be used freely in the presentation.

The author would like to thank his advisor, Uli Walther, for his encouragement during the preparation of this work.
He is also grateful to the referee for the careful reading and useful suggestions.
\section{Multiplier Ideals on Toric varieties}
Let $\KK$ be a field and $R=\KK[M \cap \sigma^{\vee}]$ be the coordinate ring of an affine normal Gorenstein toric variety. Denote $X= \Spec(R)$. 
In this case, the canonical divisor $K_X$ of $X$ is Cartier, so there exists a $u_0 \in M \cap \sigma^{\vee}$ 
such that $(u_0,n_i)=1$ where the $n_i'$s are the primitive generators of $\sigma$.
 For any monomial ideal $\fraka$ of $R$, denote
 $\N(\fraka)$ the Newton polyhedron of $\fraka$ and $\relint  \N(\fraka)$ the relative interior of $ \N(\fraka)$.
 The multiplier ideal $\mathcal{J}(\fraka)$ of $\fraka$ in $R$ admits a combinatorial description:
 \begin{prp} 
\begin{equation} \label{toricmultiplier}
    \mathcal{J} (\fraka ) = \langle \underline{x}^w \in R \mid w+u_0 \in \relint  \N(\fraka) \rangle 
    \end{equation}  
    \end{prp} 
    This is a result by Hara and Yoshida (\cite{HY03},Theorem~4.8) which is generalized by Blickle \cite{B04} to arbitrary normal toric varieties.

\section{The Example} \label{mainexample}
Consider the $3$-dimensional normal semigroup ring $R = \KK[x^2y, xy, xy^2,z]$, $\KK$ a field. 
Notice that $R$ is a complete intersection, and hence Gorenstein. Note also that \[u_0=(1,1,1).\]
Consider the following two ideals of $R$: \[\fraka = \langle x^2y^4, x^{10}y^6z^2 \rangle,\]
\[\frakb = \langle x^{12}y^7, x^{10}y^6z^2 \rangle.\]
Then $\fraka \frakb = (x^{14}y^{11}, x^{12}y^{10}z^2, x^{22}y^{13}z^2, x^{20}y^{12}z^4).$
Denote \[w_1=(14,11,0),\]
       \[w_2=(12,10,2),\]
       \[w_3=(22,13,2),\]
       \[w_4=(20,12,4).\]

Observe that the lattice point \[v=(18,12,2) \in \relint  \N(\fraka \frakb).\] 
To see this, consider the four points 
       \[v_1=w_1=(14,11,0),\]
       \[v_2=w_1+(4,2,0)=(18,13,0),\] 
       \[v_3=w_1+(2,1,4)=(16,12,4),\] 
       \[v_4=\frac1{2}(w_3+w_4)=(21,\frac{25}{2},3).\] 
They are in $\N(\fraka \frakb)$ and do not lie on a plane, namely, they are affinely independent.
Since
\[ v = \frac5{16}v_1+ \frac1{16}v_2+ \frac1{8}v_3+\frac1{2}v_4, \]
  it is in $\relint \N(\fraka \frakb)$.
 
Now, since $-u_0+v=(17,11,1)$, by \eqref{toricmultiplier}
 \[x^{17}y^{11}z \in \mathcal{J}(\fraka \frakb).\]
We claim that \[x^{17}y^{11}z \notin \mathcal{J}(\fraka) \mathcal{J}( \frakb).\]
An element in $\mathcal{J}(\fraka) \mathcal{J}( \frakb)$ is a finite sum of monomials of the form $c \cdot  \underline{x}^{\alpha} \underline{x}^{\beta}$ where $c \in \KK$, 
$\alpha,\beta \in M \cap \sigma^{\vee}$, $\alpha +u_0 \in \relint \N(\fraka)$, and $\beta +u_0 \in \relint \N(\frakb)$.
If $x^{17}y^{11}z \in \mathcal{J}(\fraka) \mathcal{J}( \frakb)$, then \[-u_0 +v=(17,11,1) = \alpha + \beta \] for some $\alpha$, $\beta$ as above. This means $v=(18,12,2)$ can be written as a sum of a lattice point $( \alpha+u_0)$ in $\relint \N(\fraka)$ and a lattice point $\beta$ in $-u_0 + \relint \N(\frakb)$. We check that this is not possible. 

\begin{figure}[h]
\[
\includegraphics[width=12cm]{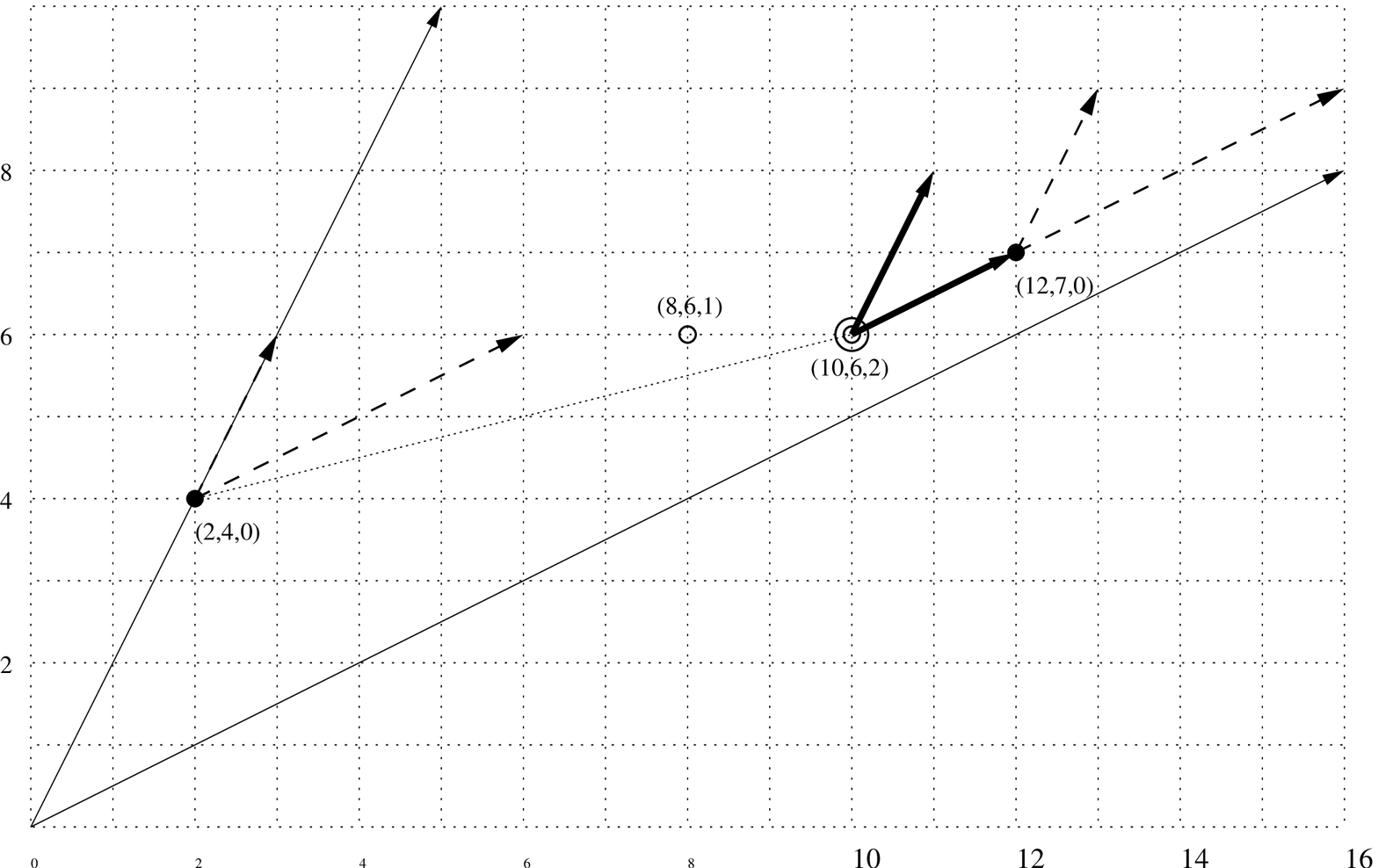}
\]
\end{figure}

Suppose $\alpha$ and $\beta$ are lattice points satisfying
$\alpha + u_0 + \beta = v = (18,12,2)$.
Write $\alpha' =\alpha+u_0 = (a_1,a_2,a_3)$ and $\beta=(b_1,b_2,b_3)$, then  \[(a_1+b_1, a_2+b_2, a_3+b_3)=v=(18,12,2) .\] 
We will show that in each case either $\alpha' \notin \relint \N(\fraka)$ or $\beta + u_0 \notin \relint \N(\frakb)$.
First, note that the Newton polyhedron $\N(\fraka)$ is the intersection of the halfspaces determined by the following  five hyperplanes: $2x-y=0$, $-x+4y=14$, $-x+2y=2$, $-x+2y+2z=6$, $z=0$.
So we have \begin{equation}\label{Na}
\relint \N(\fraka) = \{ (x,y,z) \in M | 2x-y>0, -x+4y>14,-x+2y>2,-x+2y+2z>6, z>0\}.\end{equation}
Also, $\N(\frakb)$ is the intersection of the halfspaces determined by the following four hyperplanes:
 $2x-y=14$,
 $-x+2y=2$,
 $4x-2y+3z=34$,
 $z=0$.
We have 
\begin{equation}\label{Nb}
\relint \N(\frakb)= \{(x,y,z) \in M | 2x-y>14, -x+2y>2, 4x-2y+3z>34, z>0 \}.\end{equation}
We consider the following cases.

\begin{itemize}
\item[Case I:] If $a_2 \geq 7$, then $b_2 \leq 5$ and  $\beta + u_0 \notin \relint \N(\frakb)$. To see this, 
suppose $\beta + u _0 = (b_1+1,b_2+1,b_3+1) \in \relint \N(\frakb)$. By \eqref{Nb}, $2(b_1+1)-(b_2+1) > 14$ and 
$-(b_1+1)+2(b_2+1)>2$. So 
$4(b_2+1)-4 > 2(b_1+1) >14 +(b_2+1)$ and hence $b_2 >5$, which is a contradiction.

\item[Case II:] If $a_2 \leq 4$, then $\alpha' \notin \relint \N(\fraka)$. Indeed, suppose $\alpha'=(a_1,a_2,a_3) \in \relint \N(\fraka)$. By \eqref{Na}, $2a_1-a_2 >0$ and $-a_1+4a_2 >14$. So $8a_2-28>2a_1>a_2$ and hence $a_2 >4$.

\item[Case III:] Suppose $a_2=5$ and $b_2=7$.

\begin{itemize}

\item If $a_1 \geq 6$, then $\alpha' \notin  \relint \N(\fraka)$. Indeed, suppose $\alpha'=(a_1,a_2,a_3) \in \relint \N(\fraka)$. By \eqref{Na}, $-a_1 + 4a_2> 14$ and hence $a_1 < 4a_2-14 = 6$.

\item If $a_1 \leq 5$, then $b_1 \geq 13$. This implies $\beta +u_0 \notin \relint \N(\frakb)$.
Indeed, suppose $\beta + u_0 = (b_1+1,b_2+1,b_3+1) \in \relint \N(\frakb)$. By \eqref{Nb}, $-(b_1+1)+2(b_2+1) >2$ and hence $b_1 < 2(b_1+1)-3 =13$.
\end{itemize}

\item[Case IV:] Suppose $a_2=b_2=6$. 

\begin{itemize}
\item If $b_1 \neq 10$, then $\beta+u_0 \notin \relint \N(\frakb)$. To see this, suppose $\beta+u_0 = (b_1+1,b_2+1,b_3+1) \in \relint \N(\frakb)$. By \eqref{Nb} again, $2(b_1+1)-(b_2+1) > 14$ and 
$-(b_1+1)+2(b_2+1)>2$. This forces $b_1=10$.

\item If $b_1 = 10$, then $\alpha'=(a_1,a_2,a_3)= (8,6,a_3)$ and $\beta = (b_1,b_2,b_3) = (10,6,b_3)$.
\begin{itemize}
\item If $a_3 \leq 0$, then $\alpha' \notin \relint \N(\fraka)$ by \eqref{Na}.
\item If $a_3 > 2$, then $b_3<0$. In this case, $\beta +u_0 \notin \relint \N(\frakb)$ by \eqref{Nb}.
\item If $\alpha' =(a_1,a_2,a_3)= (8,6,1)$, then $-a_1+2a_2+2a_3 = 6$. So $\alpha' \notin \relint \N(\fraka)$ by \eqref{Na}. 
\item If $\alpha' = (a_1,a_2,a_3) = (8,6,2)$, then $\beta=(b_1,b_2,b_3)=(10,6,0)$. So $4(b_1+1)-2(b_2+1)+3(b_3+1) =33<34$. Hence $\beta+u_0 \notin \relint \N(\frakb)$ by \eqref{Nb}.
\end{itemize}
\end{itemize}
\end{itemize}

\begin{rmk} We briefly explain the idea behind the example. Recall that the integral closure $\overline{I}$ of a monomial ideal $I$ in a normal toric ring $R$
is determined by $\N(I)$ (see, for example, \cite{Te04}):
\[ \overline{I} = \langle \underline{x}^w \in R \mid w \in \N(I)\rangle.
\]
So Question~\ref{subadd} is closely related to the containment $\overline{I}\cdot \overline{J} \subseteq \overline{IJ}$ for monomial ideals of $R$.
Huneke and Swanson provide a trick to construct examples where the strict containment $\overline{I} \cdot \overline{J} \subsetneq \overline{IJ}$ occur 
(see \cite{HS06} Example~1.4.9 and the remark after it). We repeat their construction here:

Choose a ring $R'$ and a pair of ideals $I'$, $J'$ in $R'$ such that \[\overline{I'}+\overline{J'} \subsetneq \overline{I'+J'}.\] 
Pick an element 
\[ r \in \overline{I'+J'} \setminus (\overline{I'}+\overline{J'}).
\]
Set $R = R'[Z]$ for some variable $Z$ over $R'$ and set
\[ I = \overline{I'}R+ZR, J=\overline{J'}R+ZR. \]
Then $I$ and $J$ are integrally closed and \[rZ \in \overline{IJ} \setminus \overline{I} \cdot \overline{J}. \]

This kind of construction doesn't always guarantee a counterexample to Question~\ref{subadd}. 
However, a suitable choice of $r$, $Z$, $R'$, $I'$, and $J'$ will do.
In our example, take \[R' = \KK[x^2y,xy,xy^2],\]
\[ r = x^8y^6,\]
\[ I'= \langle x^2y^4 \rangle,\]
\[ J'= \langle x^{12}y^7 \rangle, \]
\[ Z = x^{10}y^6z^2. \]
Then $rZ =x^{18}y^{12}z^2$ is exactly the crucial point we considered in the example.
\end{rmk}

\section{Two-dimensional case} \label{2d}
Let $R =\KK[ M \cap \sigma^{\vee}]$, $\KK$ a field, be a $2$-dimensional normal toric ring and denote $X= \Spec(R)$. 
Then there exists a primitive lattice point $w_0 \in M\cap \sigma^{\vee}$ such that $(w_0, n_i)=r \in \ZZ_{\geq 0}$
where the $n_i'$s are the primitive generators of $\sigma$. 
So the canonical divisor $K_X$ of $X$ is $\QQ$-Cartier and $R$ is $\QQ$-Gorenstein.

Set $u_0 = w_0/r$. By Theorem~4.8 in \cite{HY03}, for any monomial ideal $\fraka$ in $R$
\begin{equation} \label{Qgor}
    \mathcal{J} (\fraka ) = \langle \underline{x}^w \in R \mid w+u_0 \in \relint  \N(\fraka) \rangle. 
    \end{equation}  
The following theorem establishes the subadditivity formula on two-dimensional normal toric rings.
\begin{thm} \label{2d} For any pair of monomial ideals $\fraka$, $\frakb$ in $R$, 
\[ \mathcal{J}(\fraka \frakb) \subseteq \mathcal{J}(\fraka) \mathcal{J}(\frakb) .\]
\end{thm}
\begin{proof} Write $\fraka = \langle \underline{x}^a \mid a \in A \rangle$ and $\frakb = \langle \underline{x}^b \mid b \in B \rangle$ 
for some finite sets $A$ and $B$ in $M \cap \sigma^{\vee}$. We assume that $\{x^a | a \in A \}$ and $\{x^b | b \in B\}$ are the sets of monomial minimal generators of $\fraka$ and $\frakb$ respectively. Then $\fraka \frakb = \langle \underline{x}^{a+b} \mid a \in A \text{ and } b \in B \rangle$.
Let $\alpha_1, \dots, \alpha_k$ be the vertices of the Newton polyhedron $\N(\fraka \frakb)$ such that 
\[ \alpha_1+ \rho_1, \conv \{\alpha_1, \alpha_2 \}, \dots,
\conv \{\alpha_{k-1}, \alpha_k \}, \text{ and } \alpha_k + \rho_2 \] 
form the boundary of $\N(\fraka \frakb)$, where $\rho_1, \rho_2$ are the two rays of $\sigma^{\vee}$.
Then
\[ \N(\fraka \frakb) = \displaystyle{\bigcup^{k-1}_{i=1}} ( \conv \{\alpha_i, \alpha_{i+1}\} + \sigma^{\vee}). \]
Note also that the $\alpha_i'$s are of the form $a_i+b_i$ for some $a_i \in A$ and $b_i \in B$.
Suppose that for some $i \in \{1 ,\dots, k-1 \}$, we have $a_i \neq a_{i+1}$ and $b_i \neq b_{i+1}$. Then $a_i +b_{i+1}$ and $a_{i+1}+b_i$ lie on the boundary segment $\conv \{\alpha_i, \alpha_{i+1} \}$, 
since otherwise they lie on different sides of  $\conv \{\alpha_i, \alpha_{i+1} \}$, which is a contradiction. For any such $i$, we insert the point $a_i+b_{i+1}$ to the sequence
$\alpha_1, \dots, \alpha_k$. So we obtain a sequence, say $\beta_1=a_1'+b_1', \dots, \beta_s=a_s'+b_s'$, such that, for each $i \in \{1,\dots s-1\}$, either $a_i'=a_{i+1}'$ or
$b_i' = b_{i+1}'$, and that \[ \N(\fraka \frakb) = \displaystyle{\bigcup^{s-1}_{i=1}} ( \conv \{\beta_i, \beta_{i+1}\} + \sigma^{\vee}). \]

Now, observe that \[\relint \N(\fraka \frakb) \subseteq \displaystyle{\bigcup^{s-1}_{i=1}} ( \relint \Delta_i ),\] where 
$\Delta_i = \conv \{\beta_i, \beta_{i+1}\} +\sigma^{\vee}$. 
If $\underline{x}^p \in \mathcal{J}(\fraka \frakb)$, then by \eqref{Qgor} $p+ u_0  \in \relint \N(\fraka \frakb)$ and hence in $\relint \Delta_{i_0}$ for some $i_0$. Without loss of
generality, we may assume $a_{i_0}' = a_{i_0+1}'$. So 
\[ p+u_0 \in \relint \Delta_{i_0} = a_{i_0}' + [ \relint (\conv \{b_{i_0}', b_{i_0+1}' \} +\sigma^{\vee})] 
   \subseteq a_{i_0}'+ \relint \N(\frakb).\]
   Therefore, $p \in a_{i_0}'+[-u_0+ \relint \N(\frakb)].$ Since $a_{i_0}' +u_0 \in \relint \N(\fraka)$, by \eqref{Qgor} we conclude that $\underline{x}^p \in
   \mathcal{J}(\fraka) \mathcal{J}(\frakb)$, as desired.
\end{proof}
\begin{rmk} As one can see in the proof of Theorem \ref{2d}, the choice of $\beta_i$'s is essential. For any $\underline{x}^p \in \mathcal{J}(\fraka \frakb)$, we are able to choose 
$a \in \N(\fraka)$ such that $x^a$ is in the set of monomial minimal generators of $\fraka$ and that $p+u_0 \in a + \relint \N(\frakb)$. This cannot be extended to the higher dimensional case. From the example in 
section~\ref{mainexample}, $x^{17}y^{11}z \in  \mathcal{J}(\fraka \frakb)$ and $u_0=(1,1,1)$.
$\N(\fraka)$ is minimally generated by $x^2y^4$ and $x^{10}y^6z^2$.
But $(16,8,2)=(18,12,2)-(2,4,0)$ and $(8,6,0)=(18,12,2)-(10,6,2)$ are not in $\relint \N(\frakb)$ by \eqref{Nb}. Similarly, 
 $\N(\frakb)$ is minimally generated by $x^12y^7$ and $x^{10}y^6z^2$.
But $(6,5,2)=(18,12,2)-(12,7,0)$ and $(8,6,0)=(18,12,2)-(10,6,2)$ are not in $\relint \N(\fraka)$ by \eqref{Na}.
\end{rmk}

\bibliographystyle{amsalpha}
\bibliography{smit}

\def\cprime{$'$} \def\cfac#1{\ifmmode\setbox7\hbox{$\accent"5E#1$}\else
  \setbox7\hbox{\accent"5E#1}\penalty 10000\relax\fi\raise 1\ht7
  \hbox{\lower1.15ex\hbox to 1\wd7{\hss\accent"13\hss}}\penalty 10000
  \hskip-1\wd7\penalty 10000\box7}
  \def\cfudot#1{\ifmmode\setbox7\hbox{$\accent"5E#1$}\else
  \setbox7\hbox{\accent"5E#1}\penalty 10000\relax\fi\raise 1\ht7
  \hbox{\raise.1ex\hbox to 1\wd7{\hss.\hss}}\penalty 10000 \hskip-1\wd7\penalty
  10000\box7}
\providecommand{\bysame}{\leavevmode\hbox to3em{\hrulefill}\thinspace}
\providecommand{\MR}{\relax\ifhmode\unskip\space\fi MR }
\providecommand{\MRhref}[2]{%
  \href{http://www.ams.org/mathscinet-getitem?mr=#1}{#2}
}
\providecommand{\href}[2]{#2}
\begin{thebibliography}{DEL00}

\bibitem[Bli04]{B04}
Manuel Blickle, \emph{Multiplier ideals and modules on toric varieties}, Math.
  Z. \textbf{248} (2004), no.~1, 113--121. \MR{MR2092724 (2006a:14082)}

\bibitem[DEL00]{DEL}
Jean-Pierre Demailly, Lawrence Ein, and Robert Lazarsfeld, \emph{A
  subadditivity property of multiplier ideals}, Michigan Math. J. \textbf{48}
  (2000), 137--156, Dedicated to William Fulton on the occasion of his 60th
  birthday. \MR{MR1786484 (2002a:14016)}

\bibitem[ELS01]{ELS1}
Lawrence Ein, Robert Lazarsfeld, and Karen~E. Smith, \emph{Uniform bounds and
  symbolic powers on smooth varieties}, Invent. Math. \textbf{144} (2001),
  no.~2, 241--252. \MR{1826369 (2002b:13001)}

\bibitem[ELS03]{ELS2}
\bysame, \emph{Uniform approximation of {A}bhyankar valuation ideals in smooth
  function fields}, Amer. J. Math. \textbf{125} (2003), no.~2, 409--440.
  \MR{1963690 (2003m:13004)}

\bibitem[Ful93]{Ful}
William Fulton, \emph{Introduction to toric varieties}, Annals of Mathematics
  Studies, vol. 131, Princeton University Press, Princeton, NJ, 1993, The
  William H. Roever Lectures in Geometry. \MR{MR1234037 (94g:14028)}

\bibitem[HS06]{HS06}
Craig Huneke and Irena Swanson, \emph{Integral closure of ideals, rings, and
  modules}, London Mathematical Society Lecture Note Series, vol. 336,
  Cambridge University Press, Cambridge, 2006. \MR{MR2266432 (2008m:13013)}

\bibitem[HY03]{HY03}
Nobuo Hara and Ken-Ichi Yoshida, \emph{A generalization of tight closure and
  multiplier ideals}, Trans. Amer. Math. Soc. \textbf{355} (2003), no.~8,
  3143--3174 (electronic). \MR{MR1974679 (2004i:13003)}

\bibitem[Tei04]{Te04}
Bernard Teissier, \emph{Monomial ideals, binomial ideals, polynomial ideals},
  Trends in commutative algebra, Math. Sci. Res. Inst. Publ., vol.~51,
  Cambridge Univ. Press, Cambridge, 2004, pp.~211--246. \MR{MR2132653
  (2006c:13032)}

\bibitem[TW04]{TW04}
Shunsuke Takagi and Kei-ichi Watanabe, \emph{When does the subadditivity
  theorem for multiplier ideals hold?}, Trans. Amer. Math. Soc. \textbf{356}
  (2004), no.~10, 3951--3961 (electronic). \MR{MR2058513 (2005d:13016)}

\end{thebibliography}

\end{document}